\documentclass[11 pt]{amsart}
\usepackage[left=1.1in, right=1.1in, top=1.1in, bottom=1in]{geometry}
\usepackage{layout}
\usepackage{amssymb, amsfonts,amsthm,amsmath,latexsym,mathrsfs,mathtools,comment}
\usepackage[all]{xy}

\def\bt{\begin{thm}}
\def\et{\end{thm}}
\def\bl{\begin{lem}}
\def\el{\end{lem}}
\def\bd{\begin{defi}}
\def\ed{\end{defi}}
\def\bc{\begin{cor}}
\def\ec{\end{cor}}
\def\bp{\begin{proof}}
\def\ep{\end{proof}}
\def\br{\begin{rem}}
\def\er{\end{rem}}

\newtheorem{thm}{Theorem}[section]
\newtheorem{prop}[thm]{Proposition}
\newtheorem{lem}[thm]{Lemma}
\newtheorem{defn}[thm]{Definition}

\newtheorem{rem}[thm]{Remark}
\newtheorem{cor}[thm]{Corollary}

\numberwithin{equation}{section}

\newcommand{\la}{\langle}
\newcommand{\ra}{\rangle}
\newcommand{\C}{\Bbb{C}}

\newcommand{\R}{\Bbb{R}^m}

\newcommand{\uz}{u^{(n)}}
\newcommand{\az}{a^{(n)}}
\newcommand{\p}{\Bbb{P}}
\newcommand{\prob}{{\bf{P}}}

\newcommand{\bthm}{\begin{thm}}
\newcommand{\ethm}{\end{thm}}
\newcommand{\bstp}{\begin{stp}}
\newcommand{\estp}{\end{stp}}
\newcommand{\blemma}{\begin{lemma}}
\newcommand{\elemma}{\end{lemma}}
\newcommand{\bprop}{\begin{prop}}
\newcommand{\eprop}{\end{prop}}
\newcommand{\bpf}{\begin{pf}}
\newcommand{\epf}{\end{pf}}
\newcommand{\bdefn}{\begin{defn}}
\newcommand{\edefn}{\end{defn}}
\newcommand{\brk}{\begin{rmrk}}
\newcommand{\erk}{\end{rmrk}}
\newcommand{\bcrl}{\begin{crl}}
\newcommand{\ecrl}{\end{crl}}

\title{On Global universality for zeros of random polynomials}
\author{Turgay Bayraktar}
\address{Faculty of Engineering and Natural Sciences, Sabanc{\i} University, \.{I}stanbul, Turkey}
\email{tbayraktar@sabanciuniv.edu}
\begin{document}

\begin{abstract}
 In this work, we study asymptotic zero distribution of random multi-variable polynomials which are random linear combinations 
 $\sum_{j}a_jP_j(z)$ with i.i.d coefficients relative to a basis of orthonormal polynomials $\{P_j\}_j$ induced by a multi-circular weight function $Q$ defined on $\C^m$ satisfying suitable smoothness and growth conditions. In complex dimension $m\geq3$, we prove that $\Bbb{E}[(\log(1+|a_j|))^m]<\infty$ is a necessary and sufficient condition for normalized zero currents of random polynomials to be almost surely asymptotic to the (deterministic) extremal current $\frac{i}{\pi}\partial\overline{\partial}V_{Q}.$ In addition, in complex dimension one, we consider random linear combinations of orthonormal polynomials with respect to a regular measure in the sense of Stahl \& Totik and we prove analogous results in this setting. 
  \end{abstract}

\maketitle
\section{Introduction}
A random \textit{Kac polynomial} is of the form
$$f_n(z)=\sum_{j=0}^na_jz^j$$ where coefficients $a_j$ are independent complex Gaussian random variables of mean zero and variance one. A classical result due to Kac and Hammersley \cite{Kac,Ham} asserts that normalized zeros of Kac random polynomials of large degree tend to accumulate on the unit circle $S^1=\{|z|=1\}.$ This ensemble of random polynomials has been extensively studied (see eg. \cite{LO,HN,SV,IZ} and references therein). Recently, Ibragimov and Zaporozhets \cite{IZ} proved that for independent and identically distributed (i.i.d.) real or complex random variables $a_j$
\begin{equation}\label{IZc}
\Bbb{E}[\log(1+|a_j|)]<\infty
\end{equation} is a necessary and sufficient condition for zeros of random Kac polynomials to accumulate near the unit circle. In particular, under the condition (\ref{IZc}) asymptotic zero distribution of Kac polynomials is independent of the choice of the probability law of random coefficients. We refer to this phenomenon as \textit{global universality} for zeros of Kac polynomials.

In \cite{SZ3}, Shiffman and Zelditch remarked that it was an implicit choice of an inner product that produced the concentration of zeros of Kac polynomials around the unit circle $S^1.$ More generally, for a simply connected domain $\Omega\Subset\C$ with real analytic boundary $\partial\Omega$ and a fixed orthonormal basis (ONB) $\{P_j\}_{j=1}^{n+1}$ induced by a measure $\rho(z)|dz|$ where $\rho\in\mathcal{C}^{\omega}(\partial \Omega)$ and $|dz|$ denote arc-length, Shiffman and Zelditch proved that zeros of random polynomials
$$f_n(z)=\sum_{j=1}^{n+1}a_jP_j(z)\ \text{where}\  a_j \ \text{i.i.d standard complex Gaussians}$$  concentrate near the boundary $\partial\Omega$ as $n\to \infty.$ Furthermore, the empirical measures of zeros $$\frac{1}{n}\sum_{\{z:f_n(z)=0\}}\delta_z$$ converge weakly to the equilibrium measure $\mu_{\overline{\Omega}}$. Recall that for a non-polar compact set $K\subset \C$ the \textit{equilibrium measure} $\mu_K$ is the unique minimizer of the \textit{logarithmic energy} functional
$$\nu\to \int\int\log\frac{1}{|z-w|}d\nu(z)d\nu(w)$$ over all probability measures supported on $K.$ Later, Bloom \cite{Bloom2} observed that $\overline{\Omega}$ can be replaced by a regular compact set $K\subset \C,$ the inner product can be defined in terms of any Bernstein Markov measure (see also \cite{BloomS} for a generalization of this result to $\C^m$ for Gaussian random pluricomplex polynomials). More recently, Pritsker and Ramachandran \cite{PrK} observed that (\ref{IZc}) is a necessary and sufficient condition for zeros of random linear combinations of Szeg\"{o}, Bergman, or Faber polynomials (associated with Jordan domains bounded with analytic curves) to accumulate near the support of the corresponding equilibrium measure.\\ \indent 
 The purpose of this work is to study global universality for normalized zero currents of random multi-variable complex polynomials. Asymptotic zero distribution of multivariate random polynomials has been studied by several authors (see eg. \cite{SZ,DS3,BloomS,BloomL,B6, B7,B4}). We remark that randomization of the space of polynomials in these papers is different than that of \cite{IZ,KZ,PrK}. Namely, in the former ones each $\mathcal{P}_n$ are endowed with a $d_n:=dim(\mathcal{P}_n)$ fold product probability measure which leads to a sequence of polynomials (with $n^{th}$ coordinate has total degree at most $n$) chosen independently at random according to the $d_n$-fold product measure. On the other hand, the papers \cite{IZ,KZ,PrK} fix a random sequence of scalars for which one considers random linear combinations of a fixed basis for $\mathcal{P}_n$. We adopt the approach of \cite{IZ,KZ,PrK} in the present note.

 The setting is as follows: let $Q:\C^m\to \Bbb{R}$ be a \textit{weight function} satisfying  
 \begin{equation}\label{growth}
Q(z)\geq (1+\epsilon)\log\|z\|\ \text{for}\ \|z\|\gg 1
\end{equation} for some fixed $\epsilon>0.$ Throughout this note (unless otherwise stated), we assume that the function $Q:\C^m\to [0,\infty)$ is of class $\mathscr{C}^2$ and it is invariant under the action of the real torus $\Bbb{S}^m,$ the latter means that 
\begin{equation}\label{circular}
Q(z_1,\dots,z_m)=Q(|z_1|,\dots,|z_m|) \ \text{for all} \ (z_1,\dots,z_m)\in \Bbb{C}^m.
\end{equation}
One can define an associated weighted extremal function 
$$V_Q(z):=\sup\{u(z):u\in\mathcal{L}(\C^m), u\leq Q\ \text{on}\ \C^m\}$$ where $\mathcal{L}(\C^m)$ denotes the \textit{Lelong class} of pluri-subharmonic (psh) functions $u$ that satisfies $u(z)-\log^+\|z\|=O(1).$ We also denote by
$$\mathcal{L}^+(\C^m):=\{u\in \mathcal{L}(\C^m): u(z)\geq \log^+\|z\|+C_u\ \text{for some}\ C_u\in \Bbb{R}\}.$$
Seminal results of Siciak and Zaharyuta (see \cite{Klimek} and references therein) imply that $V_Q\in\mathcal{L}^+(\C^m)$ and that  $V_{Q}$ verifies
 \begin{equation}\label{envelope}
 V_Q(z)=\sup\{\frac{1}{\deg p}\log|p(z)|:p\ \text{is a polynomial and}\ \max_{z\in \C^m}|p(z)|e^{-deg(p)Q(z)}\leq 1\}.
 \end{equation}
Moreover, a result of Berman \cite[Proposition 2.1]{Berman1} implies that $V_Q$ is of class $\mathscr{C}^{1,1}$. 

Next, we define an inner product on the space $\mathcal{P}_n$ of multi-variable polynomials of degree at most $n$ by setting 
\begin{equation}\label{n}
\langle f_n,g_n\rangle_n:=\int_{\C^m}f_n(z)\overline{g_n(z)}e^{-2nQ(z)}dV_m(z)
\end{equation} where $dV_m$ denotes the Lebesgue measure on $\C^m$. We also let $\{P_j^n\}_{j=1}^{d_n}$ be the orthonormal basis (ONB) for $\mathcal{P}_n$ obtained by applying Gram-Schmidt algorithm in the Hilbert space $(\mathcal{P}_n, \la\cdot,\rangle_n)$ to the monomials $\{z^J\}_{|J|\leq n}$ where $J=(j_1,\dots,j_m)$ is m-multiindex and we assume that the monomials $\{z^J\}_{|J|\leq n}$ are ordered with respect to lexicographical ordering. Note that since $Q$ is $m-$circular we have $P_j^n(z)=c_J^nz^J$ for some deterministic constant $c^n_J$ and $J\in\Bbb{N}^m.$ 

Let $a_1,a_2,\dots$ be a sequence of i.i.d. real or complex random variables whose probability law denoted by $\prob$. Throughout this note, we assume that $a_j$ are non-degenerate, roughly speaking this means that $\prob[a_j=z]<1$ for every $z\in\C$ (see \S\ref{probprelim}.) A \textit{random polynomial} is of the form
$$f_n(z)=\sum_{j=1}^{d_n}a_jP_j^n(z)$$ where $d_n:=\dim(\mathcal{P}_n)={n+m\choose n}$.  
We also let $\mathcal{H}:=\cup_{n=1}^{\infty}\mathcal{P}_n$ and denote the corresponding probability space of polynomials by $(\mathcal{H},\p).$

\begin{thm}\label{main 1} Let $a_j$ be i.i.d. non-degenerate real or complex random variables satisfying
\begin{equation}\label{A}
\Bbb{E}[\big(\log(1+|a_j|)\big)^m]<\infty.
\end{equation}
If the dimension of complex Euclidean space $m \geq 3$ then almost surely in $\mathcal{H}$ $$\frac{1}{n}\log|f_n(z)|  \xrightarrow[n \to \infty]{} V_{Q}(z)$$ in $L^1_{loc}(\C^{m}).$ In particular, almost surely in $\mathcal{H}$
 $$\frac{i}{\pi}\partial\overline{\partial}(\frac{1}{n}\log|f_n(z)|)   \xrightarrow{} \frac{i}{\pi}\partial\overline{\partial}V_{Q}(z)$$ in the sense of currents as $n\to \infty.$

Furthermore, for all dimensions $m\geq 1,$ we have convergence in probability
$$\frac{i}{\pi}\partial\overline{\partial}(\frac{1}{n}\log|f_n(z)|)   \xrightarrow[]{} \frac{i}{\pi}\partial\overline{\partial}V_{Q}(z)$$ in the sense of currents as $n\to \infty.$ 
 \end{thm}

Note that Theorem \ref{main 1} provides an optimal condition on random coefficients for a random version of Siciak-Zaharyuta theorem in this context (cf. \cite{Bloom1,B6,BloomL,B7}). In the univariate case we have $\frac{i}{\pi}\partial\overline{\partial}=\frac{1}{2\pi}\Delta$ where $\Delta$ denotes the Laplacian and we denote the corresponding \textit{equilibrium measure} by $\mu_Q:=\frac{i}{\pi}\partial\overline{\partial}V_Q.$ An important example is  $Q(z)=\frac{|z|^2}{2}$ and $\mu_{Q}=\frac{1}{\pi}1_{\Bbb{D}}dz$ where $\Bbb{D}$ denotes closed the unit disc in the complex plane \cite[pp 245]{SaffTotik}. Then a routine calculation shows that
$$P^{n}_j(z)=\sqrt{\frac{n^j}{2\pi j!}}z^j\ \text{for}\ j=0,1,\dots,n$$  form an ONB for $\mathcal{P}_n.$ A random \textit{Weyl polynomial} is of the form
$$W_n(z)=\sum_{j=0}^na_j\sqrt{\frac{n^j}{ j!}}z^j.$$ In particular, Theorem \ref{main 1} generalizes a special case of \cite[Theorem 2.5]{KZ} to the several complex variables. 

Let us denote the Euclidean volume in $\C^{m}$ by $Vol_{2m}$ and for an open set  $U\subset \C^m,$ we define 
$$\mathcal{V}_{U}:=\frac{1}{(m-1)!}\int_{U} \frac{i}{\pi}\partial\overline{\partial}V_{Q}\wedge (\frac{i}{\pi}\partial\overline{\partial}\|z\|^2)^{m-1}.$$ Next result indicates that in higher dimensions the condition $(\ref{A})$ is also necessary for zero divisors of random polynomials to be almost surely equidistributed with the extremal current
$ \frac{i}{\pi}\partial\overline{\partial}V_{Q}$.
\begin{thm}\label{main}
Let $a_j$ be i.i.d. non-degenerate real or complex valued random variables and assume that the dimension of complex Euclidean space $m \geq 3.$ The logarithmic moment 
\begin{equation*}\label{A1}
\Bbb{E}[\big(\log(1+|a_j|)\big)^m]<\infty
\end{equation*}
if and only if
\begin{equation}\label{2}
\p\Big\{\{f_n\}_{n\geq0}: \lim_{n\to \infty}\frac1nVol_{2m-2}(Z_{f_n}\cap U)=\mathcal{V}_{U}\Big\}=1
\end{equation} 
for every open set $U\Subset (\C^*)^m$ such that $\partial U$ has zero Lebesgue measure.
\end{thm}
Note that when $m=1$ the volume $Vol_{2m-2}(Z_{f_n}\cap U)$  becomes the number of zeros of $f_n$ in $U$ which we denote by $$\mathcal{N}_n(U,f_n):=\#\{z\in U:f_n(z)=0\}.$$ The following result is an immediate consequence of Theorem \ref{main 1} together with Theorem \ref{main} and provides a weak universality result for zeros of univariate random polynomials: 

\begin{cor}\label{cor} Let $a_j$ be i.i.d. non-degenerate real or complex valued random variables. If the logarithmic moment
$$\Bbb{E}[\log(1+|a_j|)]<\infty$$ 
then for every $\epsilon>0$
\begin{equation}\label{cp} \lim_{n\to \infty}Prob_n\Big\{f_n: |\frac{1}{n}\mathcal{N}_n(U,f_n)-\mu_{Q}(U)\Big|\geq \epsilon\}=0
\end{equation} 
for every open set $U\Subset \C^*$ such that $\partial U$ has zero Lebesgue measure.
\end{cor}

We remark that the condition (\ref{cp}) is called \textit{convergence in probability} in the context of probability theory. Moreover, (\ref{cp}) is equivalent to the following statement: for every subsequence $n_k$ of positive integers there exists a further subsequence $n_{k_j} $ such that $\frac{1}{n_{k_j} }\mathcal{N}_{n_{k_j} }(U,f_{n_{k_j}} )\to \mu_{Q}(U)$ with probability one in $\mathcal{H}.$

Next, we consider \textit{random elliptic polynomials} which are of the form
$$G_n(z)=\sum_{|J|=n}a_J{n \choose J}^{\frac12}z^J$$
where ${n \choose J}=\frac{n!}{(n-|J|)! j_1!\dots j_m!}$ and  $a_J$ are non-degenerate i.i.d. random variables. 

 Let us denote by
 $$\mathcal{M}_{U}:=\frac{1}{(m-1)!}\int_{U}\frac{i}{2\pi}\partial\overline{\partial}(\log(1+\|z\|^2))\wedge (\frac{i}{\pi}\partial\overline{\partial}\|z\|^2)^{m-1}. $$
The following result is an analogue of Theorem \ref{main} in the present setting (see \S\ref{elliptic} for details):

\begin{thm}\label{SU}
Let $a_j$ be i.i.d. non-degenerate real or complex valued random variables and assume that the dimension of complex Euclidean space $m \geq 3.$ The logarithmic moment 
\begin{equation*}
\Bbb{E}[\big(\log(1+|a_j|)\big)^m]<\infty
\end{equation*}
if and only if the zero loci of elliptic polynomials satisfy
\begin{equation}\label{3}
\p\Big\{\{G_n\}_{n\geq0}: \lim_{n\to \infty}\frac1nVol_{2m-2}(Z_{G_n}\cap U)=\mathcal{M}_{U}\Big\}=1
\end{equation} 
for every open set $U\Subset (\C^*)^m$ such that $\partial U$ has zero Lebesgue measure.

\end{thm}

Finally, we consider random linear combinations of univariate orthonormal polynomials of regular asymptotic behavior (cf. \cite[\S 3]{StahlTotik}). Orthogonal polynomials of regular $n^{th}$ root asymptotic behavior are natural generalizations of classical orthogonal polynomials on the real line. More precisely, let $\mu$ be a measure Borel measure with compact support $S_{\mu}\subset \C.$ We assume that the support $S_{\mu}$ contains infinitely many points and its logarithmic capacity $Cap(S_{\mu})>0.$ We let $\Omega:=\overline{\C}\setminus S_{\mu}$ and $g_{\Omega}(z,\infty)$ denotes the Green function with logarithmic pole at infinity. Then the equilibrium measure of the support $S_{\mu}$ is given by $\nu_{S_{\mu}}:=\Delta g_{\Omega}(z,\infty).$ We say that $\Omega$ is \textit{regular} if $g(z,\infty)\equiv 0$ on $S_{\mu}.$ It is well know that if $\Omega$ is regular then $g(z,\infty)$ is continuous on $\C.$ Next, we define the inner product induced by $\mu:$
$$\la f,g \ra:=\int_{\C}f(z)\overline{g(z)}d\mu $$ on the space of polynomials $\mathcal{P}_n.$ Then one can find uniquely defined orthonormal polynomials
 $$P_n^{\mu}(z)=\gamma_n(\mu)z^n+\cdots,\ \text{where}\ \gamma_n(\mu)>0\ \text{and}\ n\in\Bbb{N}.$$ We say that $\mu$ is \textit{regular}, denoted by $\mu\in \textbf{Reg}$, if
\begin{equation}
\lim_{n\to \infty}\gamma_n(\mu)^{1/n}=\frac{1}{Cap(S_{\mu})}.
\end{equation}
 For a fixed $\mu\in \textbf{Reg},$ we consider random linear combinations of orthonormal polynomials 
$$f_n(z)=\sum_{j=0}^na_jP^{\mu}_j(z)$$
and we obtain the following generalization:

\begin{thm}\label{onp}
Let $\mu\in \textbf{Reg}$ such that $\Omega:=\overline{\C}\setminus S_{\mu}$ is connected and regular. Assume that the convex hull $Co(S_{\mu})$ has Lebesgue measure zero (hence, $Co(S_{\mu})$ is a line segment). If the logarithmic moment
$$\Bbb{E}[\log(1+|a_j|)]<\infty$$ 
then for every $\epsilon>0$ 
\begin{equation*} \lim_{n\to \infty}Prob_n\Big\{f_n: |\frac{1}{n}\mathcal{N}_n(U,f_n)-\nu_{S_{\mu}}(U)\Big|\geq \epsilon\}=0
\end{equation*} 
for every open set $U\Subset \C^*$ such that $\partial U$ has zero Lebesgue measure. 
\end{thm}

 We remark that if $\mu$ is a Bernstein-Markov measure with compact support in $\C$ then $\mu\in\textbf{Reg}$ (\cite[Proposition 3.4]{Bloom1}). In particular, any Bernstein-Markov measure $\mu$ supported on a compact subset of the real line falls in the framework of Theorem \ref{onp}. The latter class contains classical orthogonal polynomials such as Chebyshev or Jacobi polynomials.

\section*{Acknowledgment}
We thank N. Levenberg and T. Bloom for their comments on an earlier version of this manuscript. We are grateful to N. Levenberg for pointing out Theorem \ref{onp} falls in the framework of this work.
\section{Background}\label{prelim}
\subsection{Probabilistic preliminaries}\label{probprelim} For a complex (respectively real) random variable $\eta$ we let $\prob$ denote its probability law and denote its concentration function by
$$\mathcal{Q}(\eta,r):=\sup_{z\in \Bbb{C}}\prob[\eta\in B(z,r)]$$ where $B(z,r)$ denotes the Euclidean ball (respectively interval) centered at $z$ and of radius $r>0.$ We say that $\eta$ is \textit{non-degenerate} if $\mathcal{Q}(\eta,r)<1$ for some $r>0.$ If $\eta$ and $\xi$ are independent complex random variables and $r,c>0$ then we have 
\begin{equation}\label{levi}
\mathcal{Q}(\eta+\xi,r)\leq \min\{\mathcal{Q}(\eta,r),\mathcal{Q}(\xi,r)\}\ \text{and} \ \mathcal{Q}(c\zeta,r)=\mathcal{Q}(\zeta,\frac{r}{c}).
\end{equation}

Let $a_1,a_2,\dots $ be independent and identically distributed (real or complex valued) random variables. The following lemma is standard in the literature and  it will be useful in the sequel.
\begin{lem}\label{lem1} Let $a_j$ be a sequence of i.i.d. real or complex valued random variables for $j=1,2,\dots$
\begin{itemize}
\item[(i)] If $\Bbb{E}[\big(\log(1+|a_j|)\big)^m]<\infty$ then for each $\epsilon>0$ almost surely
\begin{equation}\label{id}
|a_j|< e^{\sqrt[m]{\epsilon j}}
\end{equation} 
for sufficiently large $j.$
\item[(ii)] If $\Bbb{E}[(\log(1+|a_j|))^m]=\infty$ then almost surely
$$ \limsup_{j\to \infty}|a_j|^{\frac1j}=\infty.$$
\end{itemize}
\end{lem}
\begin{proof}
For a non-negative random variable $X$ we have
\begin{equation}\label{ineq1}
\sum_{j=1}^{\infty}\prob[X\geq j]\leq \Bbb{E}[X]\leq 1+\sum_{j=1}^{\infty}\prob[X\geq j].
\end{equation} 
Letting $X=\frac{1}{\epsilon}(\log(1+|a_1|))^m$ and using the assumption that $a_j$ are identically distributed, we obtain
$$\sum_{j=1}^{\infty}\prob[a_j\in \C: |a_j|\geq e^{\sqrt[m]{j\epsilon}}]<\infty.$$ 
 Hence, by independence of $a_j$'s and Borel-Cantelli lemma we have almost surely
$$|a_j|<e^{\sqrt[m]{j\epsilon}}$$ for sufficiently large $j$.\\ \indent
For (ii), we define the event $A^M_j:=\{a_j\in\C:|a_j|^{\frac{m}{j}}\geq M\}$ where $M>1$ is fixed. Then by (\ref{ineq1}) 
$$\sum_{j=1}^{\infty}\prob_n[A^M_j]=\infty$$ and second Borel-Cantelli lemma implies that almost surely $|a_j|^{\frac{m}{j}}\geq M$ for infinitely many values of $j.$  Now, we let $M_n>0$ be a sequence such that $M_n\uparrow \infty.$ Then by previous argument the event
$$F_n:=\{|a_j|^{\frac{m}{j}}\geq M_n\ \text{for infinitely many}\ j\}$$  has probability one. Thus letting 
$F=\cap_{n=1}^{\infty}F_n$ has also probability one and (ii) follows.
\end{proof}

\subsection{Pluripotential Theory}
\subsubsection{Global extremal function} Let $\Sigma\subset \C^m$ be a closed set. Recall that an \textit{admissible weight function} $Q:\C^m\to \Bbb{R}$ is a lower semi-continuous function that satisfies 
\begin{enumerate} 
\item$ \{z\in\Sigma: Q(z)<\infty\}\ \text{is not pluripolar}$
\item $\displaystyle \lim_{\|z\|\to \infty}(Q(z)-\log\|z\|)=\infty$ if $\Sigma$ is unbounded.
\end{enumerate}
The \textit{weighted extremal function} associated to the pair $(\Sigma,Q)$ is defined by
\begin{equation}
V_{\Sigma,Q}=\sup\{u(z):u\in\mathcal{L}(\C^m), u\leq Q\ \text{on}\ \Sigma\}.
\end{equation}
If $\Sigma=\C^m$ and $Q$ is an admissible weight function we write $V_Q$ for short. We also let $V_{\Sigma,Q}^*$ denote the upper semi-continuous regularization of $V_{\Sigma,Q}$ that is $V_{\Sigma,Q}^*(z):=\displaystyle\limsup_{\zeta\to z}V_{\Sigma,Q}(\zeta).$  It is well known that $V_{\Sigma,Q}^*\in \mathcal{L}^+(\C^m)$ (see \cite[Appendix B]{SaffTotik}). Moreover, for an admissible weight function $Q$ the set $$\{z\in \C^m: V_{\Sigma,Q}(z)<V_{\Sigma,Q}^*(z)\}$$ is pluripolar. We also remark that when $Q\equiv 0$ and $\Sigma$ is a non-pluripolar compact set the function $V_ {\Sigma}^*$ is nothing but the pluricomplex Green function of $\Sigma$ (see \cite[\S 5]{Klimek}). We let $B(r)$ denote the ball in $\C^m$ centered at the origin and with radius $r>0.$ Then it is well known \cite[Appendix B]{SaffTotik} that for sufficiently large r 
\begin{equation}
V_Q=V_{B(r),Q}\ \text{on}\ \C^m
\end{equation} for every admissible weight function $Q.$ It also follows from a result of Siciak \cite[Proposition 2.16]{Siciak} that if $Q$ is a continuous admissible weight function then $V_Q=V_Q^*$ on $\C^m.$ We refer the reader to the manuscript \cite[Appendix B]{SaffTotik} for further properties of the weighted global extremal function. 
 
\subsubsection{Bergman kernel asymptotics}\label{scvprelim} In the sequel we will assume that $Q:\C^m\to\Bbb{R}$ is a $\mathscr{C}^2$ weight function satisfying (\ref{growth}) and (\ref{circular}). The \textit{Bergman kernel} for the Hilbert space of weighted polynomials $\mathcal{P}_n$ may be defined as
 $$S_n(z,w):=\sum_{j=1}^{d_n}P_j^{n}(z)\overline{P_j^{n}(w)}$$ where $\{P_j^n\}_{j=1}^{d_n}$ is an ONB for $\mathcal{P}_n$ as in the introduction. The restriction of the Bergman kernel over the diagonal is given by
 $$S_n(z,z)=\sum_{j=1}^{d_n}|P^{n}_j(z)|^2.$$
It is well known \cite[\S6]{BloomL} (cf. \cite{Berman1,B4}) that 
$$\frac{1}{2n}\log S_n(z,z)\to V_Q(z)\ \text{locally uniformly on}\ \C^m.$$ 
 
\section{Proofs}\label{proofs}

\begin{proof}[Proof of Theorem \ref{main 1}]
By \cite[Proposition 4.4]{BloomL} it is enough to prove that almost surely in $\mathcal{H},$ for any subsequence $I$ of positive integers 
$$(\limsup_{n\in I}\frac1n\log|f_n(z)|)^*=V_{Q}(z)$$ for all $z\in \C^m.$ To this end we fix a subsequence $I$ of positive integers.

\textbf{Step 1: Proof of upper bound.}
Note that by Lemma \ref{lem1} for each $\epsilon>0$ there exists $j_0\in\Bbb{N}$ such that almost surely
$$\sum_{j=j_0}^{d_n}|a_j|^2\leq d_ne^{2\sqrt[m]{\epsilon d_n}}.$$
Then using $d_n=O(n^m)$ and by Cauchy-Schwarz inequality almost surely in $\mathcal{H}$
\begin{eqnarray*}
\limsup_{n\in I}\frac1n\log|f_n(z)| &=& \limsup_{n\in I}\big(\frac1n \log\frac{|f_n(z)|}{\sqrt{S_n(z,z)}}+\frac{1}{2n}\log S_n(z,z)\big)\\
&\leq& \limsup_{n\to \infty}\big(\frac{1}{2n}\log(\sum_{j=1}^{d_n}|a_j|^2) +\frac{1}{2n}\log S_n(z,z)\big)\\
&\leq& \epsilon+V_{Q}(z)
\end{eqnarray*} on $\C^m.$ Thus, it follows from \cite[Lemma 2.1]{Bloom1} that
$(\displaystyle\limsup_{n\in I}\frac1n\log|f_n(z)|)^*\in\mathcal{L}(\C^m)$ and
\begin{equation}\label{limsup}
F(z):=(\limsup_{n\in I}\frac1n\log|f_n(z)|)^*\leq V_{Q}(z)
\end{equation} holds on $\C^m$ almost surely in $\mathcal{H}.$\\ 

\textbf{Step 2: Proof of lower bound.} In order to get the lower bound first we prove the following lemma which is a generalization of \cite[Proposition 2.1]{B8}:

\begin{lem}\label{orthogonal}
For every $\epsilon>0$ and $z\in (\C^*)^m$ there exists $\delta>0$ such that for sufficiently large $n\in\Bbb{N}$
$$\#\{j\in\{1,\dots,d_n\}: P^n_j(z)>e^{n(V_{Q}(z)-3\epsilon)}\}\geq \delta d_n.$$
\end{lem}
\begin{proof}
We denote the probability measures $\mu_n:=\frac{1}{b_n}e^{-2nQ(z)}dV_m$ where the normalizing constants $b_n:=\int_{\C^m}e^{-2nQ(z)}dV_m$.  It follows that the sequence of measures $\{\mu_n\}_{n=1}^{\infty}$ satisfies large deviation principle (LDP) on $\C^m$ with the rate function $\mathcal{I}(z)=2[Q(z)-\displaystyle\inf_{w\in\C^m}Q(w)]$ (see e.g. \cite[1.1.5]{DeSt}). More precisely, for $A\subset \C^m$ letting 
$$\mathcal{I}(A):=\inf_{z\in A}\mathcal{I}(z)$$
we have
$$\limsup_{n\to \infty}\frac1n\log\mu_n(K)\leq -\mathcal{I}(K)\ \text{and}\ \liminf_{n\to\infty}\frac1n\log\mu_n(U)\geq -\mathcal{I}(U) $$
for every closed set $K\subset \C^m$ and every open set $U\subset \C^m.$ 

Next, we define $$c^n_{nT}:=(\int_{\C^m}|z^T|^{2n}e^{-2nQ(z)}dV_m)^{-\frac12}$$ where $T\in [0,1]^m$ is a multi-index and $z^T=z_1^{t_1}\cdots z_m^{t_m}$. Then by Varadhan's lemma \cite[Theorem 2.1.10]{DeSt} and (\ref{growth}), for every such $T=(t_1,\dots,t_m)$
\begin{eqnarray*}
-\lim_{n\to \infty}\frac1n\log c^n_{nT} &=&\sup_{r\in\Bbb{R}^m_+}(\sum_{j=1}^mt_j\log r_j-Q(r_1,\dots,r_m))\\
&=& \sup_{S\in \R}(\la S,T\ra-Q(e^{s_1},\dots,e^{s_m}))\\
&=:&u(T).
\end{eqnarray*}
Let us denote by 
$$\Phi(S):=Q(e^{s_1},\dots,e^{s_m})$$ where $S=(s_1,\dots,s_m)\in \R$ and Legendre-Fenchel transform of $\Phi$ is by definition given by
\begin{eqnarray*}
\Phi^{\star}(T):&=&\sup_{S\in \R}(\la S,T\ra-\Phi(S))\\
&=&\sup_{S\in \R_{\geq0}}(\la S,T\ra-\Phi(S)).
\end{eqnarray*}
where the second equality follows from $Q\geq 0$.  Since $u(T)=\Phi^{\star}(T)$ for $T\in [0,1]^m$ the function $u(T)$ is a lower-semicontinuous convex on $[0,1]^m$. 

On the other hand, denoting by $\Psi(S):=V_Q(e^{s_1},\dots,e^{s_m})$ since $\Psi$ is a $\mathcal{C}^{1,1}$ convex function we have
$$\Psi(S)=\Psi^{\star\star}(S).$$ Thus, for every $\epsilon>0$ and $S\in \R$ there exists $T_0\in R^m_{\geq0}$ such that
$$\Psi(S)-\epsilon<\langle S,T_0\rangle-\Psi^{\star}(T_0)\leq \langle S,T_0\rangle-\Phi^{\star}(T_0)$$ where the latter inequality follows from the inequality $V_Q\leq Q$ on $\C^m.$ Moreover, it follows from \cite[Theorem 23.5]{Rockafellar} and $V_Q\in \mathcal{C}^{1,1}(\C^m)$ that
$T_0=\nabla\Psi(S)$ and hence by using $V_Q\in \mathcal{L}$ we conclude that $T_0\in [0,1]^m.$ Thus, for every $\epsilon>0$ and $S\in \R$ there exists $T_0\in [0,1]^m$ such that
$$\la S,T_0\ra-u(T_0)>V_Q(e^{s_1},\dots,e^{s_m})-\epsilon$$ and by lower-semicontinuity of $u$ there exists a product of intervals $\mathcal{J}\subset [0,1]^m$ containing $T_0$ such that the Lebesgue measure $|\mathcal{J}|>0$ and
$$\la S,T\ra-u(T)>V_Q(e^{s_1},\dots,e^{s_m})-2\epsilon\ \ \text{for every}\ T\in \mathcal{J}.$$ Now, for fixed $z\in(\C^*)^m$ letting $S=(\log|z_1|,\dots,\log|z_m|)$ then for sufficiently large $n$ we have
$$\frac1n \log (c^n_{Tn}|z^T|^{n})>V_{Q}(z)-3\epsilon$$ for every $T\in \mathcal{J}.$ 
Finally, letting $\mathcal{J}_n:=\{J\in\Bbb{N}^m: |J|\leq n\ \text{and}\ \frac{1}{n}J\in \mathcal{J}\}$ where $\frac1nJ:=(\frac{j_1}{n},\dots,\frac{j_m}{n})$ we see that for sufficiently large $n$ we have $$\#\mathcal{J}_n\geq \frac{d_n}{2}|\mathcal{J}|$$ where $|\mathcal{J}|$ denotes Lebesgue measure of $\mathcal{J}\subset \R.$ 
\end{proof}

Now, we turn back to proof of the lower bound. For fixed $z\in(\C^*)^m$ and for every $\epsilon>0$ by Lemma \ref{orthogonal}  there exists a product  interval $\mathcal{J}\subset [0,1]^m$ such that
 $$P_j^n(z)>e^{n(V_{Q}(z)-\epsilon)}$$ where $P_j^n(z)=C_J^nz^J$ and $J\in \mathcal{J}_n:=\{|J|\leq n: \frac{1}{n}J\in \mathcal{J}\}$. 
 Next, we define the random variables
 $$X_n:=\sum_{j\in \mathcal{J}_n}a_j\alpha_j\ \text{and}\ Y_n:=\sum_{j\not\in \mathcal{J}_n}a_j\alpha_j$$
where $$\alpha_j:=e^{-n(V_{Q}(z)-\epsilon)}P_j^n(z).$$ Then by (\ref{levi}) and sufficiently large $n$ we have
 \begin{equation}\label{es1}
 Prob_n[f_n:|f_n(z)|<e^{n(V_{Q}(z)-2\epsilon)}]\leq \mathcal{Q}(X_n+Y_n, e^{-\epsilon n})\leq \mathcal{Q}(X_n,e^{-\epsilon n}).
 \end{equation}
 Now, it follows from Kolmogorov-Rogozin inequality \cite{Esseen} and $\alpha_j>1$ that
 \begin{equation}\label{es2} \mathcal{Q}(X_n,e^{-\epsilon n})\leq C_1(\sum_{J\in \mathcal{J}_n}(1- \mathcal{Q}(a_j\alpha_j,e^{-\epsilon n}))^{-\frac12}\leq C_2 |\mathcal{J}_n|^{-\frac12}\leq C_3(d_n)^{-\frac12}.
 \end{equation}
Hence combining (\ref{es1}) and (\ref{es2}) we obtain: for every $z\in (\C^*)^m$ there exists $C_{\epsilon}>0$ such that
 \begin{equation}\label{es3}
  Prob_n[f_n:\frac1n\log|f_n(z)|<V_{Q}(z)-\epsilon]\leq \frac{C_{\epsilon}}{\sqrt{n^m}}.
 \end{equation}
Since $m\geq 3,$ it follows from Borel-Cantelli lemma and (\ref{es3}) that with probability one in $\mathcal{H}$
\begin{equation}\label{liminf}
\liminf_{n\to \infty}\frac1n\log|f_n(z)|\geq V_Q(z).
\end{equation}
 Thus, we conclude that for each $z\in(\C^*)^m$ there exits a subset $\mathcal{C}_z\subset\mathcal{H}$ of probability one such that that for every sequence $\{f_n\}_{n\in\Bbb{N}}\in \mathcal{C}_z$  
 \begin{equation}\label{lim}
F(z)=(\limsup_{n\in I}\frac1n\log|f_n(z)|)^*=V_Q(z)
 \end{equation}
 Next, we fix a countable dense subset $D:=\{z_j\}_{j\in\Bbb{N}}$ in $\C^m$ such that $z_j\in(\C^*)^m$ and (\ref{lim}) holds. Then, we define
$$\mathcal{C}:=\cap_{j=1}^{\infty}\mathcal{C}_{z_j}.$$ Note that $\mathcal{C}\subset \mathcal{H}$ is also of probability one. Since $V_{Q}(z)$ is continuous on $\C^m$ we have 
$$V_{Q}(z)=\lim_{z_j\in D, z_j\to z}V_{Q}(z_j)\leq \limsup_{z_j\in D,z_j\to z}F(z_j)\leq F(z)$$ where the second inequality follows from (\ref{liminf}) and the last one follows from upper-semicontinuity of $F(z).$  We deduce that for every $\{f_n\}_{n\in\Bbb{N}}\in \mathcal{C}$
$$F(z)=V_Q(z)$$ for every $z\in(\C^*)^m.$  Since $\{z\in\C^m: z_1\cdots z_m=0\}$ has Lebesgue measure zero, by a well-known property of psh functions we conclude that $$F(z)=V_Q(z)$$ for every $z\in \C^m.$ This completes the proof for dimensions $m\geq3$.

On the other hand, it follows from \cite[Proposition 4.4]{BloomL}, Step 1, (\ref{es3}) and the preceding argument that for every  $\epsilon>0,$ open set $U\Subset \C^m$ and sufficiently large $n$
$$Prob_n[f_n\in\mathcal{P}_n:  \|\frac1n\log|f_n|-V_Q\|_{L^1(U)}\geq \epsilon]\leq \frac{C_{\epsilon}}{\sqrt{n^m}}$$
which gives the second assertion.
\end{proof}

\begin{proof}[Proof of Theorem \ref{main}]

First, we prove that (\ref{A}) is a sufficient condition for (\ref{2}). We fix an open set $U\Subset (\C^*)^m$ such that $\partial U$ has zero Lebesgue measure. Let us denote by
$$\Theta:=\frac{1}{(m-1)!}\frac{i}{\pi}\partial\overline{\partial}V_{Q}\wedge (\frac{i}{2}\partial\bar{\partial}\|z\|^2)^{m-1}.$$ For $\delta>0$ arbitrary, we fix real valued smooth functions $\varphi_1,\varphi_2$ such that $0\leq \varphi_1\leq \chi_U\leq \varphi_2\leq 1$ and 
$$\int_U\Theta-\delta\leq \int_{\C^m}\varphi_1\Theta\leq \int_{\C^m}\varphi_2\Theta\leq \int_{\overline{U}}\Theta+\delta.$$  Now, letting 
$$\psi_j:=\frac{\varphi_j}{(m-1)!}(\frac{i}{2}\partial\bar{\partial}\|z\|^2)^{m-1}$$ for $j=1,2$ by Wirtinger's theorem we have
$$Vol_{2m-2}(Z_{f_n}\cap U)\leq \int_{Z_{f_n}}\psi_2.$$ 
Then by Theorem \ref{main 1} 
\begin{eqnarray*}
\limsup_{n\to \infty}\frac1n Vol_{2m-2}(Z_{f_n}\cap U) &\leq &\int_{\C^m} \varphi_2\Theta\\
&\leq& \int_{\overline{U}}\Theta+\delta.
\end{eqnarray*}
Similarly one can obtain 
$$\liminf_{n\to \infty}\frac1nVol_{2m-2}(Z_{f_n}\cap U)\geq \int_{U}\Theta-\delta.$$ Since $\delta>0$ is arbitrary the assertion follows.

Next, we prove that (\ref{A}) is a necessary condition for (\ref{2}). We will prove the assertion by contradiction. Assume that 
 $$\Bbb{E}[\big(\log(1+|a_j|)\big)^m]=\infty.$$ 
 By assumption $U\Subset (\C^*)^m$ so we have $0<b_n:=\min_{j=1,\dots,d_n}\inf_{z\in U}|P^n_{j}(z)|.$ For $\epsilon>0$ small we let
 $$t_n:=\big(\frac{e^{n(M_Q+\epsilon)}}{b_n}\big)^m$$ 
 where $M_Q:=\sup_{\overline{U}}V_{Q}.$ 
 Then by the argument in the proof of Lemma \ref{lem1} (ii) for each $n\in \Bbb{N}_+$ the set
 $$F_n:=\{|a_{j}|^{m/j}\geq t_n\ \text{for infinitely many}\ j\}$$ has probability one. This implies that
 $$F:=\cap_{n=1}^{\infty} F_n$$ has also probability one. %The point is that $\liminf\frac1n\log b_n>-\infty.$%
Thus, we may assume that for  infinitely many values of $n$ there exists $j_n\in\{1,\dots,d_n\}$ such that 
 \begin{equation}\label{ff}
\max_{j=1,\dots ,d_n} |a_j|^{\frac1j}= |a_{j_n}|^{\frac{1}{j_n}}\ \text{and} \ |a_{j_n}|\geq t_n^{j_n/m}.
 \end{equation}
For simplicity of notation let us assume $j_n=d_n.$   
Now, we will show that the random polynomial $f_n(z)=\sum_{j=1}^{d_n}a_jP_j^{n}(z)$ has no zeros in $U$ for infinitely many values of $n$.  Denoting $a':=(a_j)_{j=1}^{d_n-1},$ by Cauchy-Schwarz inequality, uniform convergence of the Bergman kernel on $\overline{U}$ and (\ref{ff}) we have 
  \begin{eqnarray*}
 |\sum_{j=1}^{d_n-1}a_jP_j^{n}(z)| 
 &\leq& \|a'\| S_n(z,z)^{\frac12} \\
 &\leq & \sqrt{d_n}|a_{d_n}|^{\frac{d_n-1}{d_n}} \exp(n(V_{Q}(z)+\frac{\epsilon}{2}))\\
 &\leq &  |a_{d_n}|^{\frac{d_n-1}{d_n}} \exp(n(M_Q+\epsilon))\\
 &= & \frac{ \exp(n(M_Q+\epsilon))}{|a_{d_n}|^{\frac{1}{d_n}}} |a_{d_n}|\\
 &<& b_n|a_{d_n}| 
 \end{eqnarray*}
 for infinitely many values of $n$. Hence,
 $$\sup_{z\in U} |\sum_{j=1}^{d_n-1}a_jP_j^{n}(z)| <\inf_{z\in U}|a_{d_n}P_{d_n}^n(z)|.$$ 
\end{proof}

\section{Generalizations and Concluding remarks}
\subsection{Elliptic Polynomials}\label{elliptic} Recall that a \textit{random elliptic polynomial} in $\C^m$ is of the form
$$G_n(z)=\sum_{|J|\leq n}a_J{n \choose J}^{\frac12}z^J$$
where ${n \choose J}=\frac{n!}{(n-|J|)! j_1!\dots j_m!}$ and  $a_J$ are non-degenerate i.i.d. random variables. These polynomials induced by taking $Q(z)=\frac12\log(1+\|z\|^2)$ i.e. the potential of the standard Fubini-Study K\"ahler metric on the complex projective space $\C\Bbb{P}^m.$ In this case, the scaled monomials ${N \choose J}^{\frac12}z^J$ form an ONB with respect to the inner product 
\begin{equation*}
\langle F_n,G_n\rangle_n:=\int_{\C^m}F_n(z)\overline{G_n(z)}\frac{dV_m(z)}{(1+\|z\|^2)^{n+m+1}}
\end{equation*}
 Moreover, since $Q(z)$ is itself a Lelong class of psh function the weighted extremal function in this setting is given by $$V_Q(z)=Q(z)=\frac12\log(1+\|z\|^2).$$ Specializing further, if the coefficients $a_J$ are standard i.i.d. complex Gaussians this ensemble is known as $SU(m+1)$ polynomials and their zero distribution was studied extensively among others by \cite{BBL,SZ}. 
 
 \begin{proof}[Proof of Theorem \ref{SU}]
 Since the proof is very similar to that of Theorems \ref{main 1} and \ref{main} we explain the modifications in the present setting. 

By \cite[Proposition 4.4]{BloomL} it is enough to prove that almost surely in $\mathcal{H},$ for any subsequence $I$ of positive integers 
$$F(z):=(\limsup_{n\in I}\frac1n\log|f_n(z)|)^*=V_{Q}(z)\ \text{for all}\ z\in \C^m $$ 

In order to prove  the upper bound $F(z)\leq V_Q(z),$ we use the same argument as in Thorem \ref{main 1} together with the Bergman kernel asymptotics. Namely,
letting $S_n(z,z):=\sum_{|J|\leq n} {n \choose J}|z^{2J}| $ a routine calculation gives 
$$\frac{1}{2n}\log S_n(z,z)\to \frac12\log(1+\|z\|^2)$$ locally uniformly on $\C^m$ (see eg. \cite{SZ}).
On the other hand, for the lower bound (\ref{liminf}), we need an analogue of Lemma \ref{orthogonal}. Note that $Q(z)=\frac12\log(1+\|z\|^2)$ is a multi-circular weight function whose infimum is 0 attained at $z=0.$ Then proceeding as in the proof Lemma \ref{orthogonal}, one can show that the sequence of measures $\mu_n:=\frac{1}{a_n}e^{-2nQ(z)}dV_m$ verifies a LDP with rate function $\mathcal{I}(z)=2Q(z).$ This result and Kolmogorov-Rogozin inequality allow us to prove an analogue of (\ref{es3}) in the present setting. This together with the argument in the first part of the proof of Theorem \ref{main} finish the proof of sufficiency of (\ref{A}).  In order to prove necessity, we use the Bergman kernel asymptotics and we apply the same argument as in the second part of the proof of Theorem \ref{main}. 
 \end{proof}

\subsection{Regular Orthonormal Polynomials}
\begin{comment}
In order to prove Theorem \ref{onp} we will use the following result of Grothmann:
\begin{thm}\cite{Grothmann}
Let $E\subset \C$ be a compact set such that $\C\setminus E$ is connected and regular. Let $f_n(z)$ be a sequence of polynomials of degree $n$ verifying the following conditions:
\begin{enumerate}
\item \begin{equation}\label{c1} \limsup_{n\to \infty}\big(\sup_{z\in E}\frac1n\log|f_n(z)|\big)\leq0\end{equation}
\item For every compact set $S\subset Int(E)$ we have
\begin{equation} \label{c2}\lim_{n\to\infty}\frac{1}{n}\sum_{\{z\in\C:f_n(z)=0\}}\delta_z(S)=0\end{equation}
\item There is a compact set $K\subset \overline{\C}\setminus E$ such that
\begin{equation}\label{c3} \liminf_{n\to \infty}\big[\max_{z\in K}\big(\frac1n\log|f_n(z)|-g_{\Omega}(z,\infty)\big)\big]\geq 0\end{equation}
\end{enumerate}
Then $$\frac1n\sum_{\{z\in\C:f_n(z)=0\}}\delta_z\to\mu_E$$ weakly as $n\to\infty.$
\end{thm}
\end{comment}

\begin{proof}[Proof of Theorem \ref{onp}] We proceed as in the proof of Theorems \ref{main 1} and \ref{main}. To this end we fix a subsequence $n_k$ of positive integers. It follows from  \cite[Theorem 3.1(ii) ]{StahlTotik} that
\begin{equation}\label{st}
\lim_{n\to\infty}\frac1n\log |P^{\mu}_n(z)|= g_{\Omega}(z,\infty)
\end{equation} holds locally uniformly on $\C\setminus Co(S_{\mu}).$ Denoting the Bergman kernel by
$$S_n(z,z):=\sum_{j=0}^{n}|P_j^{\mu}(z)|^2$$ we infer that
$$\frac{1}{2n}\log S_n(z,z)\to g_{\Omega}(z,\infty)$$ locally uniformly on $\C\setminus Co(S_{\mu}).$ Thus, by Lemma \ref{lem1} and Cauchy-Schwarz inequality almost surely in $\mathcal{H}$ we have
$$\limsup_{n_k\to \infty}\frac{1}{n_k}\log|f_{n_k}(z)|\leq g_{\Omega}(z,\infty)$$ for every $z\in \C\setminus Co(S_{\mu}).$  

In order to prove the lower bound, we use the local uniform convergence (\ref{st}) which replaces Lemma \ref{orthogonal}. This in turn together with Kolmogorov-Rogozin inequality give
$$Prob_n[f_n:\frac1n\log|f_n(z)|<g_{\Omega}(z,\infty)-\epsilon]\leq \frac{C_{\epsilon}}{\sqrt{n}}$$ for every $z\in\C^*\setminus Co(S_{\mu}).$ Then applying the argument in Theorem \ref{main} using the assumption $Co(S_{\mu})$ has Lebesgue measure zero we obtain the assertion.
\end{proof}

\subsection{Almost sure convergence in lower dimensions}
 In order to get almost sure convergence in Theorems \ref{main 1} and \ref{main} for complex dimensions $m\leq 2$ we need a stronger form of Kolmogorov-Rogozin inequality. More precisely, for a fixed unit vector $\uz\in\C^n,$ i.i.d. real or complex random variables $a_j$ for $ j=1,\dots ,n$ and $\epsilon\geq 0$ we consider the \textit{small ball probability}
$$p_{\epsilon}(\uz):=\prob_n[\{a^{(n)}:|\la \az,\uz\ra|\leq\epsilon\}]$$ 
where $\prob_n$ is the product probability measure induced by the law of $a_j's$ and $\la\az,\uz\ra:=\sum_{j=1}^na_j\uz_j.$ In order to obtain the lower bound in Theorem \ref{main 1} we need for every $\epsilon>0$
\begin{equation}
\sum_{n\geq 1}p_{e^{-\epsilon n}}(u^{(d_n)})<\infty 
\end{equation} for every unit vector $u^{(d_n)}\in\C^{d_n}$.

We remark that if the random variables $a_j$ are standard (real or complex) Gaussians then the probability $p_{\epsilon}(\uz)\sim \epsilon.$ In particular, $p_{\epsilon}(u^{(n)})$ does not depend on the direction of the vector $\uz.$  However, for most other distributions, $p_{\epsilon}(\uz)$ does depend on the direction of $\uz.$ For instance if $a_j$ are Bernoulli random variables (i.e. taking values $\pm1$ with probability $\frac12$) then $p_{0}((1,1,0,\dots,0))=\frac12$ on the other hand,  $p_{0}((1,1,\dots,1))\sim n^{-\frac12}.$ Determining small ball probabilities is a classical theme in probability theory. We refer the reader to the manuscripts \cite{FriS,TaoV,RV1,RV2} and references therein for more details.\\ \indent  
 Another interesting problem is to find a necessary and sufficient condition for almost sure convergence of normalized zero currents when the space of polynomials $\mathcal{P}_n$ is endowed with $d_n$-fold product probability measure. A sufficient condition was obtained in \cite{B6}. Namely, let $a_j^n$ be iid random variables whose probability $\prob$ has a bounded density and logarithmically decaying tails i.e.
 \begin{equation}\label{tailc}
 \prob\{a_j\in\C:\log |a_j|>R\}= O(R^{-\rho})\ \text{as}\ R\to \infty\ \text{for some}\ \rho>m+1.
 \end{equation}
 We consider random polynomials of the form $f_n(z)=\sum_{j=1}^{d_n}a^n_jP^{n}_j(z)$. If (\ref{tailc}) holds then almost surely normalized zero currents $\frac1n[Z_{f_n}]$ converges weakly to the extremal current $\frac{i}{\pi}\partial\overline{\partial}V_Q.$ 
\subsubsection{Higher codimensions} In \cite[Theorem 1.2]{B6} (see also \cite{B7}) it is proved that if the coefficients of random polynomials $f_n(z)=\sum_{j=1}^{d_n}a^n_jP^n_j(z)$ are i.i.d random variables whose distribution law verifies (\ref{tailc}) then almost surely normalized empirical measure of zeros
$$\frac{1}{n^m}\sum_{\{z\in\C^m:f^1_n(z)=\dots=f^m_n(z)=0\}}\delta_z$$ of $m$ i.i.d. random polynomials $f_n^1,\dots,f_n^m$  converges weakly to the weighted equilibrium measure $(\frac{i}{\pi}\partial\overline{\partial}V_{\Sigma,Q}^*)^m$. In the present paper, we have observed that for codimension one we no longer need $a_j$ to have a density with respect to Lebesgue measure. For instance, $a_j$ can be discrete such as Bernoulli random variables. It would be interesting to know if \cite[Theorem 1.2]{B6} or a weaker form of it (eg. convergence with high probability) also generalizes to the setting of  discrete random variables.

%%%%%%%%%%%%%%%%%%%%%%%%%


\begin{thebibliography}{XXXXX}

\bibitem[Bay]{B8}
T.~Bayraktar, \emph{Expected number of real roots for random linear
  combinations of orthogonal polynomials associated with radial weights},
  Potential Anal. DOI: 10.1007/s11118-017-9643-9.

\bibitem[Bay16]{B6}
T.~Bayraktar, \emph{Equidistribution of zeros of random holomorphic sections},
  Indiana Univ. Math. J. \textbf{65} (2016), no.~5, 1759--1793. \MR{3571446}

\bibitem[Bay17a]{B4}
T.~Bayraktar, \emph{Asymptotic normality of linear statistics of zeros of
  random polynomials}, Proc. Amer. Math. Soc. \textbf{145} (2017), no.~7,
  2917--2929. \MR{3637941}

\bibitem[Bay17b]{B7}
T.~Bayraktar, \emph{Zero distribution of random sparse polynomials}, Michigan
  Math. J. \textbf{66} (2017), no.~2, 389--419. \MR{3657224}

\bibitem[BBL96]{BBL}
E.~Bogomolny, O.~Bohigas, and P.~Leboeuf, \emph{Quantum chaotic dynamics and
  random polynomials}, J. Statist. Phys. \textbf{85} (1996), no.~5-6, 639--679.
  \MR{1418808}

\bibitem[Ber09]{Berman1}
R.~Berman, \emph{Bergman kernels for weighted polynomials and weighted
  equilibrium measures of {$\Bbb C^n$}}, Indiana Univ. Math. J. \textbf{58}
  (2009), no.~4, 1921--1946. \MR{2542983 (2010g:32003)}

\bibitem[BL15]{BloomL}
T.~Bloom and N.~Levenberg, \emph{Random {P}olynomials and
  {P}luripotential-{T}heoretic {E}xtremal {F}unctions}, Potential Anal.
  \textbf{42} (2015), no.~2, 311--334. \MR{3306686}

\bibitem[Blo05]{Bloom1}
T.~Bloom, \emph{Random polynomials and {G}reen functions}, Int. Math. Res. Not.
  (2005), no.~28, 1689--1708. \MR{2172337 (2006k:32069)}

\bibitem[Blo07]{Bloom2}
T.~Bloom, \emph{Random polynomials and (pluri)potential theory}, Ann. Polon.
  Math. \textbf{91} (2007), no.~2-3, 131--141. \MR{2337837 (2008g:32051)}

\bibitem[BS07]{BloomS}
T.~Bloom and B.~Shiffman, \emph{Zeros of random polynomials on {$\Bbb C^m$}},
  Math. Res. Lett. \textbf{14} (2007), no.~3, 469--479. \MR{2318650
  (2008f:32009)}

\bibitem[DS89]{DeSt}
J.-D. Deuschel and D.~W. Stroock, \emph{Large deviations}, Pure and Applied
  Mathematics, vol. 137, Academic Press, Inc., Boston, MA, 1989. \MR{997938}

\bibitem[DS06]{DS3}
T.-C. Dinh and N.~Sibony, \emph{Distribution des valeurs de transformations
  m\'eromorphes et applications}, Comment. Math. Helv. \textbf{81} (2006),
  no.~1, 221--258. \MR{2208805 (2007i:32017)}

\bibitem[Ess68]{Esseen}
C.~G. Esseen, \emph{On the concentration function of a sum of independent
  random variables}, Z. Wahrscheinlichkeitstheorie und Verw. Gebiete \textbf{9}
  (1968), 290--308. \MR{0231419}

\bibitem[FS07]{FriS}
O.~Friedland and S.~Sodin, \emph{Bounds on the concentration function in terms
  of the {D}iophantine approximation}, C. R. Math. Acad. Sci. Paris
  \textbf{345} (2007), no.~9, 513--518. \MR{2375113 (2008j:60050)}

\bibitem[Ham56]{Ham}
J.~M. Hammersley, \emph{The zeros of a random polynomial}, Proceedings of the
  {T}hird {B}erkeley {S}ymposium on {M}athematical {S}tatistics and
  {P}robability, 1954--1955, vol. {II} (Berkeley and Los Angeles), University
  of California Press, 1956, pp.~89--111. \MR{0084888 (18,941c)}

\bibitem[HN08]{HN}
C.~P. Hughes and A.~Nikeghbali, \emph{The zeros of random polynomials cluster
  uniformly near the unit circle}, Compos. Math. \textbf{144} (2008), no.~3,
  734--746. \MR{2422348 (2009c:30017)}

\bibitem[IZ13]{IZ}
I.~Ibragimov and D.~Zaporozhets, \emph{On distribution of zeros of random
  polynomials in complex plane}, Prokhorov and Contemporary Probability Theory,
  Springer, 2013, pp.~303--323.

\bibitem[Kac43]{Kac}
M.~Kac, \emph{On the average number of real roots of a random algebraic
  equation}, Bull. Amer. Math. Soc. \textbf{49} (1943), 314--320. \MR{0007812
  (4,196d)}

\bibitem[Kli91]{Klimek}
M.~Klimek, \emph{Pluripotential theory}, London Mathematical Society
  Monographs. New Series, vol.~6, The Clarendon Press, Oxford University Press,
  New York, 1991, Oxford Science Publications. \MR{1150978 (93h:32021)}

\bibitem[KZ14]{KZ}
Z.~Kabluchko and D.~Zaporozhets, \emph{Asymptotic distribution of complex zeros
  of random analytic functions}, Ann. Probab. \textbf{42} (2014), no.~4,
  1374--1395. \MR{3262481}

\bibitem[LO43]{LO}
J.~E. Littlewood and A.~C. Offord, \emph{On the number of real roots of a
  random algebraic equation. {III}}, Rec. Math. [Mat. Sbornik] N.S.
  \textbf{12(54)} (1943), 277--286. \MR{0009656 (5,179h)}

\bibitem[PR17]{PrK}
I.~Pritsker and K.~Ramachandran, \emph{Equidistribution of zeros of random
  polynomials}, J. Approx. Theory \textbf{215} (2017), 106--117. \MR{3600134}

\bibitem[Roc97]{Rockafellar}
R.~T. Rockafellar, \emph{Convex analysis}, Princeton Landmarks in Mathematics,
  Princeton University Press, Princeton, NJ, 1997, Reprint of the 1970
  original, Princeton Paperbacks. \MR{1451876}

\bibitem[RV08]{RV1}
M.~Rudelson and R.~Vershynin, \emph{The {L}ittlewood-{O}fford problem and
  invertibility of random matrices}, Adv. Math. \textbf{218} (2008), no.~2,
  600--633. \MR{2407948 (2010g:60048)}

\bibitem[RV09]{RV2}
M.~Rudelson and R.~Vershynin, \emph{Smallest singular value of a random rectangular matrix}, Comm.
  Pure Appl. Math. \textbf{62} (2009), no.~12, 1707--1739. \MR{2569075
  (2011a:60034)}

\bibitem[Sic81]{Siciak}
J.~Siciak, \emph{Extremal plurisubharmonic functions in {${\bf C}^{n}$}}, Ann.
  Polon. Math. \textbf{39} (1981), 175--211. \MR{617459 (83e:32018)}

\bibitem[ST92]{StahlTotik}
H.~Stahl and V.~Totik, \emph{General orthogonal polynomials}, Encyclopedia of
  Mathematics and its Applications, vol.~43, Cambridge University Press,
  Cambridge, 1992. \MR{1163828}

\bibitem[ST97]{SaffTotik}
E.~B. Saff and V.~Totik, \emph{Logarithmic potentials with external fields},
  Grundlehren der Mathematischen Wissenschaften [Fundamental Principles of
  Mathematical Sciences], vol. 316, Springer-Verlag, Berlin, 1997, Appendix B
  by Thomas Bloom. \MR{1485778 (99h:31001)}

\bibitem[SV95]{SV}
L.~A. Shepp and R.~J. Vanderbei, \emph{The complex zeros of random
  polynomials}, Trans. Amer. Math. Soc. \textbf{347} (1995), no.~11,
  4365--4384. \MR{1308023 (96a:30006)}

\bibitem[SZ99]{SZ}
B.~Shiffman and S.~Zelditch, \emph{Distribution of zeros of random and quantum
  chaotic sections of positive line bundles}, Comm. Math. Phys. \textbf{200}
  (1999), no.~3, 661--683. \MR{1675133 (2001j:32018)}

\bibitem[SZ03]{SZ3}
B.~Shiffman and S.~Zelditch, \emph{Equilibrium distribution of zeros of random polynomials}, Int.
  Math. Res. Not. (2003), no.~1, 25--49. \MR{1935565 (2003h:60075)}

\bibitem[TV09]{TaoV}
T.~Tao and V.~H. Vu, \emph{Inverse {L}ittlewood-{O}fford theorems and the
  condition number of random discrete matrices}, Ann. of Math. (2) \textbf{169}
  (2009), no.~2, 595--632. \MR{2480613 (2010j:60110)}

%\bibitem[Zel98]{Zelditch}
%S.~Zelditch, \emph{Szeg\"{o} kernels and a theorem of tian}, Internat. Math.
%  Res. Notices (1998), no.~6, 317--331. \MR{1616718}



\end{thebibliography}
\end{document}